\documentclass[11pt]{amsart}
\usepackage{amsmath}
\usepackage{amsfonts}
\usepackage{amssymb}
\usepackage{amsthm}

\usepackage{graphicx}

\usepackage{times}

\newtheorem{theorem}{Theorem}[section]
\newtheorem{proposition}[theorem]{Proposition}
\newtheorem{lemma}[theorem]{Lemma}
\newtheorem{defn}[theorem]{Definition}
\newtheorem{cll}[theorem]{Corollary}

\theoremstyle{remark}
\newtheorem{remark}[theorem]{Remark}

\title{Stationary properties of the Gauss-Galerkin QM\lowercase{o}M truncation of MV-SDEs}
\author{A. Alecio}

\date{\today}
\begin{document}

\begin{abstract}
The Gauss Galerkin Method/Quadrature method of moments (GG-QMoM) closure scheme, introduced by Dawson \cite{DA}, closes a truncated set of moment equations of an SDE by a Galerkin approximation of its law in the space of probability measures.

Here, results are presented on stationary solutions of the closed equations, irrespective of the number of moments retained (thus not dependent on the convergence theorem). These are applied to polynomial MV-SDEs, with explicit dependence on its moments in the drift, which can possess multiple stationary solutions. 

Particularly, we show as $\sigma\downarrow0$ there are as many stationary solutions as extrema of the potential (critically, not just the minima), preserving the bifurcation diagram of the full equations.  Further, a scaling property of solutions is proven, allowing changes of stability to be directly probed. Finally, this is applied to the GG-QMoM truncation of Dawson-Shiino model, whose stationary measures and change in stability has been previously studied.
\end{abstract}


\maketitle
For the one-dimensional time-homogenous stochastic differential equation
\begin{equation}\label{proto}
dX_t = a(X_t) dt + \sigma b(X_t)dW_t,\quad\sigma\in\mathbb{R}, |b|>0
\end{equation}
a typical problem would be to calculate its aggregate properties from this path-wise description. This work is concerned with a method, first introduced in \cite{DA}, to calculate the moments of the law of process (\ref{proto}).

Suppose SDE (\ref{proto}) has polynomial coefficients, and possesses a unique, possibly weak, solution. The evolution of moments of its law can be calculated applying It\^o's lemma with $x^i$, and taking the expected value:
\[
\dot m_i=\mathbb{E}(dX^i_t)=i\mathbb{E}[ a(X_t)X_t^{i-1}dt+\sigma^2\frac{(i-1)}{2}b(X_t)X_t^{i-2}dt+\sigma b(X_t)X_t^{i-1}dB_t]
\]
The expectation of the martingale (final) term is null and, as $a,\,b$ are polynomials in $X_t$, we are left with a polynomial in moments of the law. 
Equivalently \begin{equation}\dot m_i=\int\mathsf{L}x^n\rho(t)\mathrm{d}x\end{equation}
where $\mathsf{L}=a\partial_x+\frac{\sigma^2}{2}b^2\partial_{xx}$, the generator of process (\ref{proto}).
In this way, SDE (\ref{proto}) is converted into a denumerable set of ODEs, the \textit{moment evolution equations} (MEEs), 
\begin{equation}\label{nce}
	\dot m_i=f_i(m_1,\dots,m_{i+j}),\quad j>0
\end{equation}
which are row finite ($j<\infty$ is independent of $i$), where $f_i$ are all polynomial functions. 

In the `lower-diagonal' case, $(j=0)$, an arbitrary moment can be calculated with knowledge the lower order moments only. For instance, the MEEs of the Ornstein-Uhlenbeck process $dX_t=-\alpha X_t dt +\sigma dW_t$ are lower diagonal with explicit solution
\begin{equation}\label{true}
\begin{split}
&m_n(t)=e^{-\alpha nt}m_n(0) +\frac{\sigma^2}{2}n(n-1)\int_0^t e^{-\alpha n(t-2)}m_{n-2}(s)ds\\
\end{split}
\end{equation}
Critically, as the equation for $m_1$ is closed, the moments can be calculated to an arbitrary order.

In the general row-finite case,
it is not typically possible to find a closed expression for a finite number of moments as we did in (\ref{true}). Moment closure methods circumvent this by replacing $m_{n+j},\, j>0$ in the $n^{th}$ and lower order MEEs with \textit{closure functions} $g_{n+j}(m_1,\dots,m_n)$. The resulting lower diagonal system of equations are
\begin{equation}
    \dot m_i=f_i(m_1,\dots,m_{i},g_{i+1}(m_1,\dots,m_i),\dots, g_{i+j}(m_1,\dots,m_i))
\end{equation}
the first $n$ of which we call $n$-\textit{closed moment evolution equations} ($n$-MEEs), and its solution the $n$-\textit{closed moments}.

The choice of $g_{n+j}(m_0,\dots,m_n)$ significantly affects the behaviour of the $n$-closed moments. A number of \textit{truncation/closure schemes} have been developed \cite{bell}, including Gauss-Galerkin QMoM \cite{DA}. The priority of GG-QMoM is to close the MEEs in a way that guarantees the Hankel matrix of $n$-closed moments is positive semidefinite (a necessary condition of the moments of any probability measure), following a fotiori as the $n$-closed moments correspond to an actual probability measure.

This measure (the \textit{n}-approximant) is an approximation of the law of SDE (\ref{proto}) by a superposition of $n$-point masses with varying intensities $\beta_i$ and nodes, $x_i$. Initialised with the Gauss-Christoffel approximation of the initial measure, the moments are evolved according to the first $2n$ moment evolution equations (\ref{nce}). Higher order moments are entirely determined by the lower, closing the system.

In effect, this is a Galerkin approximation to the weak form of the Fokker-Planck equation, where the approximant is in a subset of the space of probability measures $P_M(\mathbb{R})$ for which polynomial test functions consitute a seperating set\footnote{More precisely, Definition \ref{pms}} \cite{breiman}.

With knowledge of the first $n$ moments, \cite{DA,haj2} give explicit forms for $g_{n+1}$ for $f(X_t)$ as the solution of $h_{n+1}=0$, where $f$ is a positive diffeomorphism, and $h_n$ the $n^{th}$ Hankel determinant of the law of $f(X_t)$. 

Practically, \cite{Camp} provides an efficient algorithm for the resolution of the GG-QMoM $n$-MEEs. For this work, rather than directly resorting to the moment equations, equations are derived for $\beta_i,\,x_i$ using the transformation $m_i=\sum_j\beta_jx_j^i$. The resulting Jacobian is an invertible confluent Vandermonde matrix \cite{gaut}, from which the \textit{Gauss-Galerkin equations} can be derived. 

Not only does GG-QMoM systematically yield positive definite $n-$closed moments, but the associated approximant $\tilde \rho_{n}$ converges to the law of the true solution as more moment evolution equations are retained.
\begin{theorem}[Convergence of GG-QMoM approximates: Campillo, Dawson \cite{Camp,DA}]\label{convergencethm}
	
	Suppose there exists constants $\{k_n:n\geq1\}$ and $\theta_0>0$ such that 
	\begin{align}
	m_n(t)\leq k_n\quad \mathrm{for}\quad 0\leq t\leq T \\
	\sum_{n=1}^\infty\theta^n\frac{k_n}{n!}<\infty\quad\mathrm{for\, all}\quad 0<\theta<\theta_0
	\end{align}
	for the MEEs associated to process $X$. Further, for suitable initial probability measure suppose process $X$ possesses a unique law.  
	
	Then $\tilde\rho_n(t)\Rightarrow\rho(t)$ as $N\rightarrow\infty$, where $\rho_t=\mathrm{Law}(X_t)$, in the sense of weak convergence of probability measures.
\end{theorem}
\begin{proof} This is Theorem 5.1 in \cite{DA}. See also \cite{Camp}. Both rely on the traditional trilogy of demonstrating relative compactness of approximants, identifying the limit, and proving uniqueness.
\end{proof}
\subsection{MV-SDEs and Outline}
McKean-Vlasov SDEs \cite{McKean} (MV-SDEs) are SDEs with explicit dependence on its law. MV-SDEs with explicit dependence on moments of the law and polynomial drift and diffusion such as
\begin{equation}\label{proto1}
	dX_t=\big(-\bar V^{'}(X_t)+\theta\mathbb{E}[P^{'}(X_t)]\big)dt+\sigma k(X_t)dB_t
\end{equation}
will yield row-finite MEEs, thus moment truncation schemes are equally applicable. Many foundational works in the field such as \cite{desi,kome} took this approach, and indeed has been reprised by later ones, \cite{agp}. MV-SDEs (\ref{proto1}) are known to exhibit critical behaviour - multiple stationary measures, the number of which is a function of the parameters of (\ref{proto1}). 

MV-SDE (\ref{proto1}) with the simple bistable potential and quadratic interaction ($\bar V^{'}=x^3$, $P^{'}=x$), extensively studied in \cite{dawson,shiino}, is used to test the efficacy of several truncation methods fixed number of moments ($N=4$) in \cite{valsakumar}. 

Broadly, there are two factors which favour truncation schemes to study critical behaviour of MV-SDEs (\ref{proto1}). Stationary solutions with our assumptions, are unimodal which happen to be particularly compatible with simple truncation schemes. 
Further, they all preserve low order MEEs, the first of which was demonstrated in \cite{alecio5} to be equivalent to the self-consistency equation, an auxiliary implicit integral equation stationary measures must satisfy.

In an unpublished work, the author similarly studied GG-QMoM $n$-MEEs of this model and the related model studied in \cite{alecio3}, finding it performed comparably or outperformed the schemes of \cite{valsakumar}, even in difficult parameter regimes.

Particularly, the pitchfork bifurcation diagram of attainable stationary measures was well preserved in GG-QMoM. Crucially, whilst all methods produced spurious stationary solutions, GG-QMoM was the only method accurate in the important limit $\sigma\downarrow0$, for all parameter combinations.

Much of this accuracy stems from the explicit preservation of the drift as $\sigma\downarrow 0$, which is compatible with the crucial Theorem \ref{mv-sde} (see also \cite{alecio5}). See Theorem \ref{big} Lemma \ref{cps} and Proposition \ref{mvsdeapx}.

Another factor are the parity properties of the Gauss-Galerkin equations which given that `the retention of the symmetry properties of the potential should form the guiding principle for approximations in the analysis of system' (\cite{valsakumar} p.389) is exceedingly important. See Proposition \ref{parity}.

Critical behaviour of MV-SDEs is not simply an asymptotic phenomena. Using a scaling property of explicit stationary solutions, we are able to explore properties for $\sigma$ away from 0. See Lemma \ref{scalemv} and example.

The purpose of this short work is to prove several results of predominantly stationary solutions of GG-QMoM, independent of $N$, the number of moments retained. Particularly, we focus on those which affect the bifurcation diagram of the GG-QMoM truncated MV-SDE (\ref{proto1}), though we present versions of results applicable for both SDE (\ref{proto}) and MV-SDE (\ref{proto1}) with the intention of showing GG-QMoM is particularly well-suited to MV-SDEs. 


\section{Properties of GG-QMoM}
In the introduction it was mentioned how the original Gauss-Galerkin equations of Dawson \cite{DA} were derived. These were usefully recast by Hajjafar \cite{Hajj} in terms of the Langrange interpolating polynomials, which is the form we exclusively use in this work.
\begin{defn}[Gauss Galerkin Equations]\label{gge}
    Associated to SDE (\ref{proto}), the Gauss-Galerkin equations for nodes $x_i$ and intensities $\beta_i$ are
    \begin{align}
    \beta_i\dot x_i =\sum_{k=1}^n \beta_k \mathsf L((x-x_i)l_i^2(x))|_{x_k}\\
    \label{shortbeq}\dot \beta_i=\sum_{k=1}^n \beta_k \mathsf L(l_i^2(x))|_{x_k}
    \end{align}
    where $\mathsf L$ is the generator of (\ref{proto}). In long form:
    \begin{align}
    \label{x-eq}\beta_i\dot{x_i}=\beta_ia_i+\sigma^2\Big(2\beta_ib^2_il^{'}_{ii}+\sum_{j\ne i}\beta_jb^2_j(x_j-x_i)l_{ij}^{'2} \Big)
    \\
    \label{b-eq}\dot\beta_i=2\beta_ia_il^{'}_{ii}+\sigma^2\Big(\beta_ib^2_i(l^{'2}_{ii}+l^{''}_{ii})+\sum_{j\ne i}\beta_jb^2_jl^{'2}_{ij} \Big)
    \end{align}
    with $l_{ij}=l^{'}_i(x_j)$ and $a_i=a_1(x_i)+a_2(\underline{x},\underline{\beta})$
    \end{defn}
   
    The only assumptions made at this stage come from the classic criteria for well-posedness and ergodicity of the underlying Smoluchowski-type SDEs, \cite{stroock}:  $a,\,b$ are polynomials, $b>0$ and $\lim_{|x|\rightarrow\infty}\int^x\frac{a}{b}=\infty$.\footnote{Implying, as $a$ is a polynomial, that $a$ is the derivative of a confining potential}
    
    That these equations always define a probability measure, which was clear in the original derivation of \cite{Camp,DA}, is not immediately apparent in this form. 
    Hajjafar, in \cite{Hajj}, reestablished this, along with certain other critical properties 
    
    \begin{lemma}[Hajjafar \cite{Hajj}]\label{clo}
        If $\beta_i(0)>0$ then $\beta_i(t)> 0,\,\forall t>0$ Further, if $\sum_i\beta_i(0)=1$ then $\sum_i\beta_i(t)=1,\,\forall t>0$
    \end{lemma}
    \begin{proof}
    The second assertion is simply a restatement of the zeroth moment equation. The first follows from the quasi-positivity of $\mathbf{B}\rfloor_{ik}=\mathsf{B}(l_i^2(x_k))$
    \end{proof}
    If any nodes $x_i=x_j$, (\ref{x-eq}) blows up, therefore the $x_i$s retain the ordering they possess at $t=0$.
    \begin{proposition}[Ordering of $x_i$]\label{order}
        If $x_i(0)<x_j(0)$, then $x_i(t)<x_j(t)$ 
    \end{proposition}
    
    \begin{remark}\label{symord}
    We will exclusively use the ordered labelling of the $x_i$s such that $x_i<x_j$ iff $i<j$.  For symmetric $n$-approximants, we number away from the axis of symmetry, and assign the negative index to the mirror node and associated weight. 
    
    \end{remark}
    
    Just as with the Fokker Planck equation, the Gauss-Galerkin equations respect the parity of the potential. Numbering as Remark \ref{symord},
    
    \begin{proposition}[Parity of $\tilde{\rho}_n$]\label{parity}
        Suppose drift $a(x)$ is anti-symmetric and diffusion $b(x)$ symmetric. If $\tilde\rho_n(0)$ is symmetric then so is $\tilde\rho_n(t)$.
    \end{proposition}
    \begin{proof} 
        
        Given $l_{ij}=l_{i-j}$ for symmetric points $x_i$, it is immediate that $\dot x_i=-\dot x_{-i}$ are $\beta_i=\beta_{-i}$, from which the result follows.
         \end{proof}
        
         The remaining results are concerned with the stationary Gauss-Galerkin equations (\ref{x-eq},\ref{b-eq}), where $\dot x_i= \dot \beta_i=0$.

         \begin{theorem}[Limit of Stationary Solutions, $\sigma\downarrow 0$]\label{big}$ $ \newline
            $\{\tilde\rho_{0,N}^\sigma\}_\sigma$ possesses convergent subsequences, the limit of which are necessarily of the form $\sum_i\beta_i\delta_{x_i^{*}}$.
            where $x^*_i$ is a root of the drift.
        \end{theorem} 
        \begin{proof}
            Fix any $N\in\mathbb{N}$.
            That $\{\tilde\rho_{0,N}^\sigma\}_\sigma$ is sub-sequentially convergent (in distribution) follows from tightness, which can be deduced using that the potential is confining to form a-priori estimates showing $m_2$ is bounded as $\sigma\downarrow 0$\footnote{for an explicit computation along these lines, see Appendix \ref{ape}}  

    
            Picking any convergent subsequence and eliminating sub-indexes, we denote any member as $\tilde\rho_{0,N}^{\sigma_j}$, $\lim_{j\uparrow \infty}\sigma_j=0$. For all $\sigma_j>0$, the CDF $F_{\sigma_j}$ is a piecewise constant increasing function with $N$ jumps. Its limit $F_0$ must therefore be the same with $n\leq N$ jumps, the location and size of which are the limits of those of $F_{\sigma_j}$. 
            
            With this, we can identify $\tilde\rho_{0,N}^0$ uniquely as a superposition of $n\leq N$ point masses. Consequently, their locations and weights must be limits of the locations of the comprising point masses and weights of $\tilde\rho_{0,N}^{\sigma_j}$.


            For any index $i$ such that $\lim\limits_{\sigma\downarrow0} b_i\neq 0$, $$c_i=\lim_{\sigma\downarrow 0}\sigma^2\Big(2\beta_ib^2_il^{'}_{ii}+\sum_{j\ne i}\beta_jb^2_j(x_j-x_i)l_{ij}^{'2} \Big)$$ must exist\footnote{if $c_i=\infty$ then $b_i=0$ by bounds on $m_2$ and all terms in $x_i$ cancel out of the other equations} and so (\ref{x-eq}) has the form
            \begin{equation}\label{sta} 0=\beta_ia_i+c_i(x_1,\dots,x_n)\end{equation}
            It suffices to prove $c_i=0,\,\forall i\leq n$. 
            
            If $x_i=x_j$ then $c_i=c_j$ and we can eliminate the $j^{th}$ equation for being tautological. Reciprocally,
            as the moment equations are continuous, $\tilde\rho_{0, N}^0$ satisfies the stationary moment equations 
        \begin{equation}\label{subst}
        0 = (k+1)\mathbb{E}[a(X)X^{k}] = (k+1)\sum_i \beta_i x^{k}a_i 
        \end{equation}
        Multiplying each equation (\ref{sta}) by $x^k_i$ and summing, we get the equations
        \[
        0=\sum_i\beta_ix^k_ia_i+\sum_ix^k_ic_i
        \]
        The first sum is 0 by (\ref{subst}) leaving, in matrix-vector form,
        \begin{equation}
        \begin{pmatrix}
            
                x_1 & x_2 & x_3&\cdots \\
                x^2_1 & x^2_2&x_3^2 & \cdots  \\
                \vdots  & \vdots  &\vdots &\ddots   \\
            \end{pmatrix}
            \begin{pmatrix}
                c_1\\
                c_2  \\
                \vdots 
            \end{pmatrix}
        =\underline{0}
        \end{equation}
        Having already eliminated any repeated $x_i$s, we have an invertible Vandermonde matrix,  and so $c_i=0$. 
        \end{proof}

        Next, we calculate the rate $x_i^\sigma\rightarrow x^*_i$. With a simple polynomial drift with one root, with negative gradient, and constant diffusion, $x_i-x^*\propto\sigma^\alpha$, where $\alpha$ will be determined. Without loss of generality, we take $x^*=0$.

         \begin{lemma}\label{sdescale}
            Suppose $(\underline{x},\,\underline{\beta})$ is a solution to the stationary Gauss-Galerkin equations (\ref{x-eq},\ref{b-eq}) with $\sigma b\equiv 1$ and $a=-kx^{2n-1}$. Then $(\sigma^\alpha\underline{x},\,\underline{\beta})$ is a solution to (\ref{x-eq},\ref{b-eq}) with the same drift and $b\equiv\sigma$, where $\alpha=\frac{1}{n}$
        \end{lemma}

\begin{proof}
    Replace all $x_j$ with $\sigma^{\alpha}\tilde x_j$ in stationary (\ref{x-eq}, \ref{b-eq}) and group by powers of $\sigma^\alpha$. Since every term of $l_{ij}^{'}$ and $l_{ii}^{''}$ (see Appendix) is the quotient of products of brackets of the form $(x_i-x_j)$, every term in (\ref{b-eq}) is $\mathcal{O}(1)$, and every term in (\ref{x-eq}) is $\mathcal{O}(\sigma^\alpha)$, where a simple balancing argument shows $\alpha$ solves $2-\alpha=(2n-1)\alpha$.
\end{proof}
        
In the case $n=1$, $\underline{x}$, the locations of the $N$ point masses are exactly the roots of the degree $N$ Hermite polynomial, $He_N$, and $\underline{\beta}$ the Gaussian weights. For $n>1$, the moment equations are no longer lower diagonal, and the analogous result cannot hold due to the approximated upper moments introduced.



    

We are now in a position to construct first order solutions to the Gauss-Galerkin equations showing, reciprocally to Theorem \ref{big}, that close to a particular class of zeros of the drift, there is a solution.

\begin{proposition}\label{const}(The first approximation of $x_i$)
    Suppose drift $a(x)$ has $m$ simple, decreasing roots, $x^*_{i}$.\footnote{$a^{'}(x^*_{i})< 0$}
    Then the positions of the $mk$ point masses 
    comprising $\tilde\rho_{0,mk}$ can be expanded $\sigma<<1$,  
    $$
    x_{ij}=x^*_i+\sigma h_{j} + \mathcal{O}(\sigma^2) 
    $$
    with $j\leq k,\, i\leq m$ where $h_{j}$ is $j^{th}$ root of $He_k(x).$\footnote{The Probabilist's Hermite polynomial of order $k$}


    
\end{proposition}
\begin{proof}
Informed by Theorem \ref{big} and Lemma \ref{sdescale}, we expand $(x_{ij}-x_i^*)=\sigma f_{ij,1}+\mathcal{O}(\sigma^2)$
implying, by expanding (\ref{b-eq}) that $\beta_{ij}=\beta_{ij}^0+\sigma g_{ij,1}+\mathcal{O}(\sigma^2)$. The first order solution is the pair $(f_{ij,1}, \beta_{ij}^0)$.

First we fix $m=1$. Dividing the $i^{th}$ equation by $b_i$ and expanding the drift, keeping the zeroth and first order terms,
$$
\big(1-\frac{(b^2)^{'}(x^*)}{b^2(x^*)}(x_j-x^*)\big)\big(-\frac{|a^{'}(x_i)|}{b^2(x_i)}(x_j-x^{*})\big)
$$
therefore we retain to the lowest order  
\begin{equation}\label{ddd}
    -\frac{|a^{'}(x^*)|}{b^2(x^*)}(x_j-x^{*})
\end{equation}

Similarly $\frac{b_j}{b_i}=\mathcal{O}(1)$ so we retain the stationary (\ref{x-eq}, \ref{b-eq}) with drift (\ref{ddd}), which is the just the Gauss-Galerkin equations for the Ornstein-Uhlenbeck process. Therefore $f_{j,1}=h_j$.

For $m>1$, we introduce double index $ij$, where abscissae in the $i^{\mathrm{th}}$ cluster converge to the $i^{\mathrm{th}}$ root of $a(x)$ like $x_{ij}-x^*_i=\mathcal{O}(\sigma^\alpha)$, and number the $\beta$s concordantly. 

To proceed, we need the following two observations: for nodes from different clusters, ${x_{in}-x_{jl}}=\mathcal{O}(1)$ and the number of nodes in any cluster is $k$, so for any two abscissae $\frac{R^{'}_{jl}}{R^{'}_{in}}=\mathcal{O}(1)$

For the equations for $(x_{ij},\beta_{ij})$, from the first observation, the only terms we retain in $l^{'}_{in,in}$ are those from the $i^{th}$ cluster, $\sum_{l\leq k,', l\neq n}\frac{1}{x_{in}-x_{il}}$. 

From the first and second observation terms $\sigma^2l^{'}_{il,jk}$ are an order greater than $\sigma^2l^{'}_{il,ik}$, and are consequently discarded.

Therefore any equations in a cluster only depends on the nodes and weights of that same cluster. After performing the same balancing procedure with the drift, the lowest order of drift to be retained is
$$
    -\frac{|a^{'}(x_i^*)|}{b^2(x^*_i)}(x_j-x_i^{*})
$$
which leaves us with precisely the evolution equation for $m=1$. 
The first order solution is consequently a superposition of first order solutions centered on each $x_i^*$ 
\end{proof}

\begin{remark}
    Stationary solutions converge to point masses on minima of the potential in Proposition \ref{const} while there is no such restriction in Theorem \ref{big}. 
    
    This distinction is not simply an artifact of the proof: a single point mass on a maxima is a solution to (\ref{x-eq},\ref{b-eq}). Less trivially, an existing solution can be augmented. For instance, consider the Gauss-Galerkin equations for $dX_t=X_t-X^3_t+\sigma dW_t$. Negative roots of the potential are at $\pm 1$ but a point mass placed at 0 will remain there by symmetry.
\end{remark}

The sequel is dedicated to applying the above results to the Gauss-Galerkin equations of McKean-Vlasov SDEs 
\begin{equation}\label{mvp}
	dX_t=\big(-\bar V^{'}(X_t)+\mathbb{E}[P^{'}(X_t)]\big)dt+\sigma b(X_t)dB_t
\end{equation}
That McKean-Vlasov SDEs of the form (\ref{mvp}) have row-finite moment equations which can be closed using GG-QMoM was explored in the introduction. More directly, with regard to Definition \ref{gge}, $\mathsf L^{mv}$, the generator of MV-SDE (\ref{proto1}) is itself a function of the moments of the law, which can be approximated by moments of the $n$-approximant to close the GG-QMoM equations (\ref{x-eq},\ref{b-eq}).

In the sequel, we term $\bar V - P$ the potential, $\bar V$ the effective potential. To apply the results of \cite{alecio5} we need the following assumptions. $P^{'}$ and $\bar V^{'}$ are polynomials and $\frac{1}{x^2}\int^x \frac{V^{'}}{b^2} dx>0$.\footnote{implying it is a confining potential} Additionally $\bar V^{''}\geq 0$ attaining the lower bound at a finite number of isolated points $\{\tilde{x}_i\}_i^m$ where $(\bar V-P)^{''}(\tilde{x}_i)\neq 0$

Such MV-SDEs are known to have multiple stationary measures. In this regard, we reference from section 1 of \cite{alecio5}. With the assumptions above, 

\begin{theorem}[Alecio \cite{alecio5}]\label{mv-sde}
If $(\bar V-P)^{'}$ has $N$ simple zeros, then there exists a critical threshold $\sigma_c$ such that for $\sigma<\sigma_c$, $F_\sigma(m)$ has precisely $N$ zeros.
\end{theorem}

This implies the bifurcation diagram has $N$ branches for sufficiently small $\sigma$. Something similar can be shown for the GG-QMoM equations: the zeros of the GG-QMoM drift $-\bar V^{'}_i+\sum_j\beta_j P^{'}_j$ correspond to the zeros of $V^{'}-P^{'}$ and by Theorem \ref{big}, these must be the limit points of stationary solutions. An analogue of Proposition \ref{const} is given, guaranteeing nearby solutions to all zeros. 
 
\begin{lemma}[Critical points]\label{cps}
	$\forall n \in \mathbb{N}$, the only solution of stationary drift equation $-\bar V^{'}_i+\sum_j\beta_j P^{'}_j=0$ is $x\mathbf{1}_n$, where $x$ is any solution of $\bar V^{'}-P^{'}=0$
\end{lemma}
\begin{proof} As $\bar V^{'}$ is a homeomorphism, $\forall i, j,\, x_i=x_j:=x=(\bar V^{'})^{-1}(\sum_j\beta_jP^{'}_j)^\frac{1}{3}$. 
Then substituting $\sum_j\beta_j P^{'}(x)=P^{'}(x)$ back into the drift equations, $x$ must solve $\bar{V}^{'}-P^{'}=0$.
\end{proof}

Again, analogously to Lemma \ref{sdescale}, stationary solutions to the Gauss-Galerkin equations with simple drift and constant diffusion scale with diffusion strength.

\begin{lemma}\label{scalemv}
    Suppose $(\underline{x},\,\underline{\beta})$ is a solution to (\ref{x-eq},\ref{b-eq}) with $b\equiv 1$ and $a=-kx^{2n-1}$. Then ($\sigma^\alpha\underline{x},\,\underline{\beta})$ is a solution to (\ref{x-eq},\ref{b-eq}) with $a=-kx^{2n-1}+\theta\sum_jx_j^{2m+1}$ and $b\equiv\sigma$, where $\alpha=\frac{1}{n}$
\end{lemma}
\begin{proof}
    Following Proposition \ref{parity} we again look for symmetric solutions. Then 
    $$\sum_j\beta_j (x_j-0)^{2m+1}=0$$
    The result follows by Lemma \ref{sdescale}.
\end{proof}
This scaling property will be made use of in the upcoming example. Analogously to Proposition \ref{const}, 

\begin{proposition}(The first approximation of $x_i$ for MV-SDEs)\label{mvsdeapx}
    Suppose $V^{'}$ has $n$ roots at $x^*_{i}, \,i\leq n$. Set $n_i$ such that $2n_i$ is the lowest order non-null derivative of $\bar V$ at $x^*_{i}$.
    
    Then the positions of the $k$ point masses 
    comprising $\tilde\rho_{0,k}$ can be expanded $\sigma<<1$,  
    $$
    x_{j}=x^*_i+\sigma^{\frac{1}{n_i}} h_{j} + \mathcal{O}(\sigma^{1+\frac{1}{n_i}}) 
    $$
    with $j\leq k,\, i\leq n$ where $h_{j}$ is the $j^{th}$ solution of (\ref{x-eq}), the Gauss-Galerkin equations (\ref{x-eq},\ref{b-eq}) with $a=-x^{2n_i+1}$ and $b=1$.


    
\end{proposition}
\begin{proof}
    That $n_i$ exists follows from assumptions on the effective potential. Expanding the drift in the zeroth and first order terms,
\begin{equation}
\begin{split}
\big(1-\frac{(b^2)^{'}(x^*_i)}{b^2(x^*_i)}(x_j-x^*_i)\big)&\big(-\frac{\bar V^{(2n)}(x_i^*)(x_j-x^{*}_i)}{b^2(x_i^*)}\\
&+\frac{(P^{'}(x^*_i)-\bar V^{'}(x^*_i))+\frac{1}{N}\sum_k P^{''}(x_i^*)(x_k-x^*_i)}{b^2(x_i^*)}
\big)    
\end{split}
\end{equation}

Following lemmata \ref{cps} \textit{\&} \ref{scalemv}, the terms on the second line are null so we retain to the lowest order  
\begin{equation}
    -\frac{\bar V ^{(2n_i)}(x_i^*)}{b^2(x_i^*)}(x_j-x^{*}_i)^{2n_i-1}
\end{equation}
the coefficient of which is positive by assumption. 

By lemmata \ref{sdescale} and \ref{scalemv}, the solution to the lowest order is that of Gauss Galerkin equations (\ref{x-eq},\ref{b-eq}) with $a=-\frac{\bar V^{(2n_i)}}{b^2(x_i^*)}(x-x^*)^{2n_i-1}\,\textit{\&} \,\,b=\sigma$, from which the result follows.
\end{proof}

\begin{remark}
    Notice the apparent disjunction between Proposition \ref{const} and \ref{mvsdeapx}: in \ref{mvsdeapx}, stationary solutions converge to any extrema of the potential while in \ref{const} they only converge to minima (or superpositions thereof) of the potential. In fact, there is no real tension here. By the reasoning of Proposition \ref{mvsdeapx}, 
    $a_i\approxeq c-\bar V^{'}$, with $c=P^{'}(x^*)$, which has one root, at $x^*$, corresponding to a minima of its negative primitive.

    Underpinning this caricature is the asymptotically symmetrical arrangement of nodes so that the first order terms of $P^{'}$ cancel out, like Lemma \ref{scalemv}. In a way made precise in the sequel, away from this strict arrangement $a_i$ behaves more like potential $(V-P)^{'}$, with ramifications on the stability of these stationary measures.
\end{remark}

The final aspect of critical behaviour to review is (in)stability of stationary measures. It is clear that if $F^{'}_\sigma(m_i)<0$ then $\rho_0[m_i]$, (where $x^*(m_i)=x^*_i$) is unstable. In the case of quadratic interaction (explicit dependence on $m_1$), variational methods show the converse, see \cite{shiino} for instance. 


Paralleling this calculation for the truncated system, we study differential perturbations analogous to releasing the full system from a unimodal distribution `close' to the stationary distribution. This corresponds to moving all the stationary nodes in a commensurate direction (a \textit{unidirectional} perturbation), such as translation. A calculation is then performed, sufficient to show instability, and indicative of stability.



The author in \cite{alecio5} demonstrates and exploits the equivalency of the first MEE and self-consistency equation. Perturbing the nodes as a function of $h$ we have 
\begin{equation}\label{sum}
\frac{d\tilde m_1}{dh}=\sum_i\beta_i(-V^{''}(x_i))\frac{d x_i}{dh}
\end{equation}

It is clear unidirectional perturbations - $\frac{dx_i}{dh}>0$ (respectively $<0$) - increases (decreases) $m_1$ 
At this juncture we invoke a corollary of a theorem in \cite{alecio5}
\begin{cll}
    There exists $\epsilon>0$ and $\sigma_c$ such that, $\sigma<\sigma_c$ and $i\leq n$, $$\mathrm{sign}_{x\in J_i}(-V^{''}(x))=\mathrm{sign}_{m\in x^{*-1}(J_i)}(F^{'}_\sigma(m))$$
    where $J_i=B_{x^*_i}(\epsilon)$
\end{cll}

\begin{lemma}\label{stab1}
For any unidirectional perturbation $h$ and sufficiently small $\sigma$, 
$$\mathrm{sign}(\frac{d\tilde m_1}{dh}|_{\tilde\rho_0})=\mathrm{sign}(-V^{''}(x^*_i))$$
\end{lemma}
\begin{proof}
By Theorem \ref{big} each individual term in sum (\ref{sum}) tends to $-V^{''}(x_i^*)$, so eventually all $x_i\in J$ and the result follows.
\end{proof}

\subsection{Example: Dawson-Shiino Model - Bistable Potential with Quadratic Interaction}

In this section, we focus on the quartic potential $V^{'}=x^3-x$, and $P^{'}=x$. Explicitly,
\begin{equation}
\label{sdep}
dX_t=(-X_t^3+m_1(t))dt+\sigma dW_t
\end{equation}
 with concomitant non-linear, non-local Fokker-Planck equation 
\begin{equation}
\label{fp}
\rho_t=\frac{\partial}{\partial x}\big[ [x^3-\mu]\rho+\frac{\sigma^2}{2}\rho_x\big]
\end{equation}
We summarise the critical behaviour studied in \cite{dawson,shiino}. 

\begin{proposition}[\cite{dawson,shiino}]\label{summ}$ $
    \begin{enumerate}
    \item For all $\sigma$, there is a symmetric stationary solution to (\ref{fp}), $\rho_0^0$
		\item Below the critical threshold $\sigma_c\approxeq0.956$ there are exactly 3 stationary solutions.
		\item The symmetric stationary solution becomes unstable below $\sigma_c$
	\end{enumerate}
\end{proposition}

Below $\sigma_c$, $\rho_0^0$ is unstable but there is a class of probability distributions that still converges to $\rho^0_{0}$, which is easy to characterise. 
 
\begin{proposition}[Symmetric Subspace of (\ref{sdep})]\label{symm}
    If $\rho(0)$ is symmetric then so is $\rho(t)$ (the solution to (\ref{fp}) with initial data $\rho(0)$). Further $\lim\limits_{t\rightarrow\infty}\rho(t)=\rho_{0}^0$
\end{proposition}
\begin{proof} Symmetry follows from antisymmetry of potential. Therefore $m_1\equiv0$ and solutions to (\ref{fp}) are the same as for the Fokker Planck equation associated to process $-X_t^3+\sigma dB_t$, which converges to $\rho_{0}^0$.
\end{proof}

\begin{figure}[ht]
	\centering
	\includegraphics[width=\textwidth]{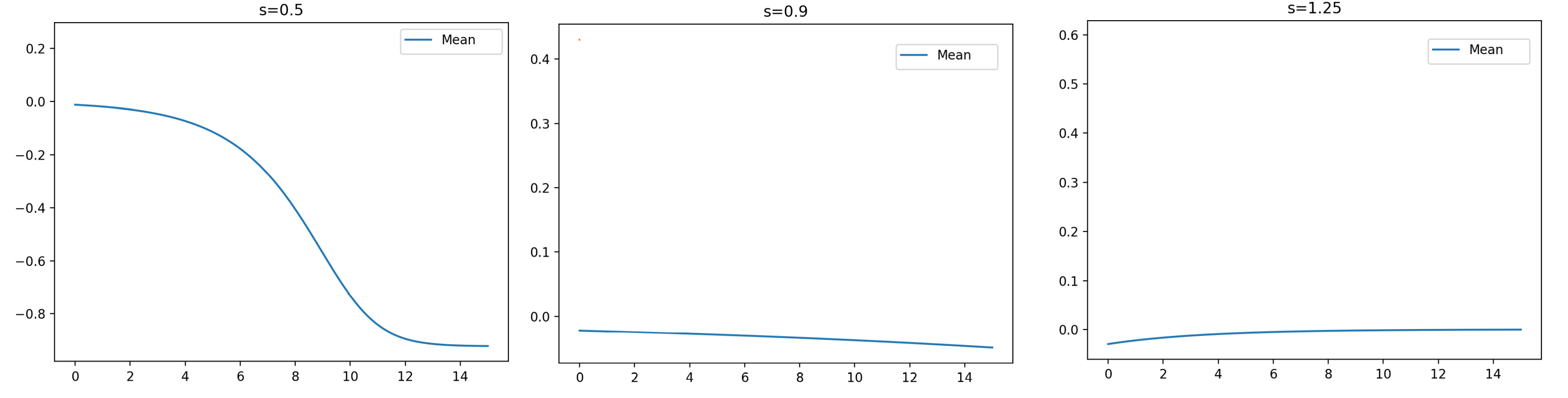}
	\caption{Gauss Galerkin Moment Closure for (\ref{sdep}):  Initialised with $\mathcal{N}(-0.1,0.1)$, $n=8$ (and changing $y$-axis scale). The equilibrium $m_1$ is approaching zero and there is a change in stability between $\sigma=0.9$ and $1.25$. The critical temperature $\sigma_c$ was determined to be $\sim0.97$ }
\end{figure}

The simple symmetric bistable potential with quadratic (Currie-Weiss) interaction is the simplest MV-SDE with well-studied critical behaviour, making it a good test case for any truncation scheme.

What makes GG-QMoM particularly appealing as a truncation scheme for this example is the guaranteed existence of a symmetric solution with, by Proposition \ref{scalemv}, an explicit representation, allowing us to directly probe its stability. 

We briefly collate the properties of stationary solutions of a GG-QMoM truncation of SDE (\ref{sdep}) in light of our results.

With the estimates proven in Appendix \ref{ape}, we can apply Theorem \ref{convergencethm} and conclude the solution of (\ref{x-eq}, \ref{b-eq}) converges to the law of (\ref{sdep}).

As $\sigma\downarrow 0$ $\tilde\rho_{n,0}$ converges to point masses centered at $\pm 1$ and $0$, according with Proposition \ref{mv-sde}.
For all $\sigma$, there exists a symmetric stationary solution to (\ref{x-eq},\ref{b-eq}), converging to a point mass centered at 0 as $\sigma\downarrow 0$

With Proposition \ref{parity} and Theorem \ref{convergencethm} we know, analogously to Proposition \ref{symm}, any symmetric initial distribution will converge to $\tilde\rho_0^0$.

From Lemma \ref{stab1} we know $\tilde\rho_0^0$ will be unstable with respect to a unidirectional perturbation, for sufficiently small $\sigma$. For this example, and a simple translation $\frac{dx_i}{dh}=1$, (\ref{sum}) is
$$
\frac{d}{dh}\sum_i\beta_i(x_i-x^3_i)=\sum_i\beta_i(1-3x^2_i)
$$
By Lemma \ref{scalemv}, this is equivalent to
$$
\sum_i\beta_i(1-3\sigma^2\tilde{x}_i^2)=1-3\sigma^2m_2^{\sigma=1}
$$
where $\tilde x_i$ are the abscissae of, and $m_2$ the second moment of the $n$-approximant when $\sigma=1$. The system is necessarily unstable when 
\begin{equation}\label{rear}
\sigma<\sqrt[2]{\frac{1}{3m^{\sigma=1}_2}}
\end{equation}
As $N$ increases, this tends to the second moment of the law of the SDE. Then (\ref{rear}) $\sim 0.993$ which is in good accord with Proposition \ref{summ} (\textit{iii}) and numerical experimentation, see Figures 1 and 2.

\begin{figure}[ht]
	\centering
	\includegraphics[width=\textwidth]{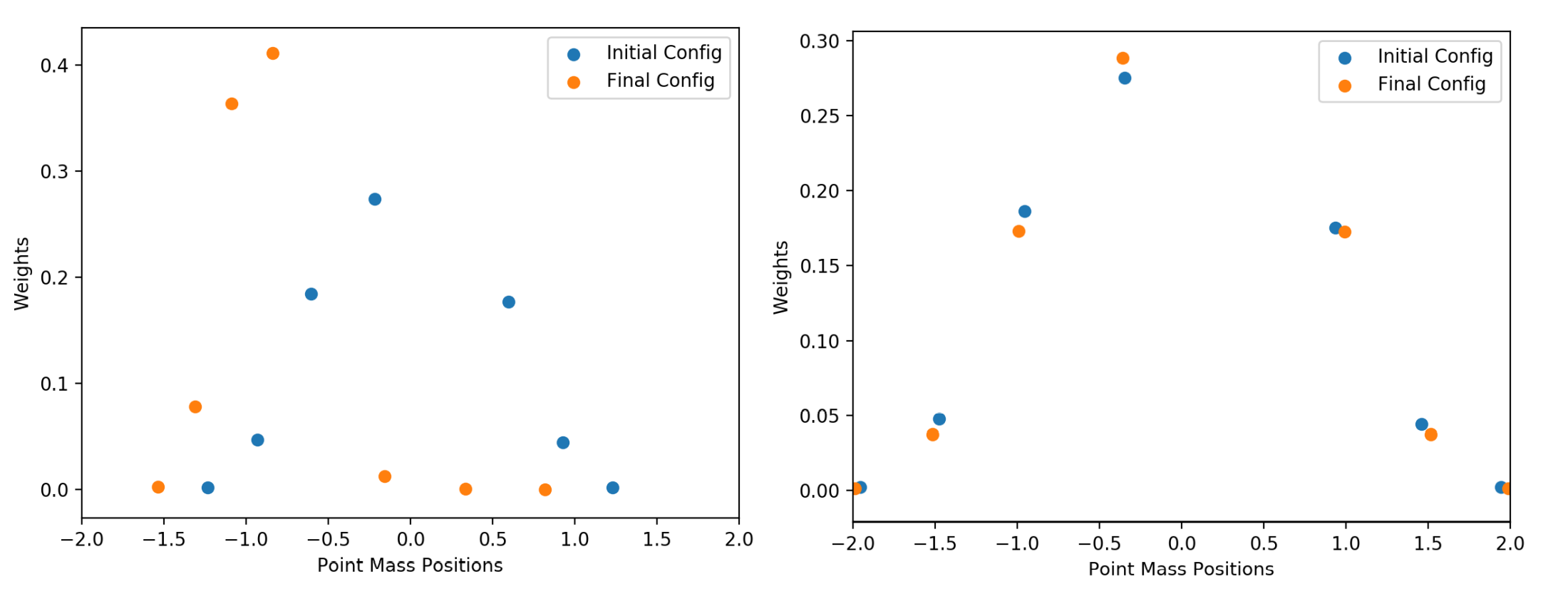}
	\caption{Gauss Galerkin Moment Closure for (\ref{sdep}):  Initialised with $\exp{(-2(x^4-0.1x))}$ $n=8$ $\sigma=0.5$ (left), $\sigma=1.25$ (right)}
\end{figure}

\appendix
\section{Langrange Interpolating Polynomials}
\begin{defn}Given a set of nodes $\{x_j:0\le i\le N,\, x_i\ne x_j\}$ 
    the $i^{th}$ Lagrange polynomial is a polynomial of $\deg(l_i)=N$ such that $l_i(x_j)=\delta_{ij}$\end{defn}
    The explicit form of this polynomial is
    $$l_i(x)=\prod_{k\ne i} \frac{x-x_k}{x_i-x_k}$$
    with first derivative
    $$l^{'}_i(x)=\sum_{k\ne i}\frac{1}{x_i-x_k}\prod_{k\ne (i,k)} \frac{x-x_m}{x_k-x_m}$$
    
    For the second, we only require
    $$l_i^{''}(x_i)=l^{''}_{ii}=(\sum_{k\ne i}\frac{1}{x_i-x_k})^2-\sum_{k\ne i}\frac{1}{(x_i-x_k)^2}
    $$
     
    Alternatively, Szeg\H o \cite{Szego} provides a conjunct representation via the base polynomial $R(x)=\prod_k(x-x_k)$. Then 
    $$l_i(x)=\frac{R(x)}{R^{'}(x_i)(x-x_i)}$$
    
    Here, we only need the derivative of the Lagrange polynomials at their nodes which is easy to express in terms of $R^{'}(x_j)=R^{'}_j$
    
    \begin{equation}
    \label{h}l^{'}_{ij}=\left\{ \begin{array}{rcl}
    \sum_{k\ne i}\frac{1}{x_i-x_k},&\quad i=j\\
    \frac{R^{'}_j}{R^{'}_i}\frac{1}{x_j-x_i},& \quad \mathrm{otherwise}
    
    \end{array}\right.
    \end{equation}

\section{A Priori Estimates for MV-SDE (\ref{sdep})}\label{ape}

We head this section by introducing the subset of probability measures with factorial-summable moments
\begin{defn}\label{pms}
	$\mathcal{P}_M(\mathbb{R})$ is a subset of probability measures	such that 
    \[\sum_{n=1}^\infty\frac{m_n}{n!}<\infty\] 
    (the moments are $\mathrm{`factorial}\, \mathrm{summable'}$) where $m_n=\int_x x^n\rho(dx)$ \\
\end{defn}
A sufficient condition for $\rho\in\mathcal{P}_M(\mathbb{R})$ is that the even moments are factorial-summable:
\begin{lemma}\label{norm}
    If $\sum_k \frac{m_{2k}}{2k!}<\infty$ then  $\sum_k \frac{|m_{k}|}{k!}<c\sum_k \frac{m_{2k}}{2k!}<\infty$.
\end{lemma}
\begin{proof} This follows from the Cauchy inequality
    $$
    |x|^{2k-1}=\frac{x^{2k}}{4k}+kx^{2k-2}
    $$
    Integrating both sides of the equality with respect to the measure, dividing by $ (2k-1)!$ and simplifying, we see $\sum\frac{|m|_{2k-1}}{(2k-1)!}\leq2\sum_k \frac{m_{2k}}{2k!}$. The result follows with $c\geq 3$. 
\end{proof}

\begin{proposition}\label{measbou}
	Given an initial distribution $\rho_0\in\mathcal{P}_M(\mathbb{R})$, then the law of process (\ref{sdep}) $\rho(t)\in\mathcal{P}_M(\mathbb{R})$ too.
\end{proposition}
\begin{proof}
    This needs two auxiliary results: lemma \ref{norm} and
    \begin{lemma}[Lyapunov]\label{vil}
        If $0<s<t$, then for random variable $X$ $$\mathbb{E}(|X|^s)^\frac{1}{s}\leq \mathbb{E}(|X|^t)^\frac{1}{t}$$
    \end{lemma}
    
    \begin{proof} Apply Jensen's inequality with $g(x)=|x|^{\frac{t}{s}}$\end{proof}

   Lemma (\ref{vil}) lets us majorise the equation as  \begin{equation}\label{bound}\dot{m}_{2k}\leq2k\big(\frac{\sigma^2}{2}m_{2k}^{1-\frac{1}{k}}+(1-\theta)m_{2k}-m_{2k}^{1+\frac{1}{k}}\big)\end{equation}
	Equation (\ref{bound}) has only one positive root, $m_{2k}^*$ from positive to negative. Therefore if, for any $t\geq0$ or $k$, $m_{2k}(t)>m_{2k}^*$,  $m_{2k}$ will decrease. 
	
    
    The roots $m_{2k}^*$ can be determined by solving the concomitant quadratic equation 
    $$\frac{\sigma^2}{2}+(1-\theta)y-y^2=0$$
    where $m_2k^{\frac{1}{k}}=y$. The solution $\sim \sigma$, so $\sum_k\frac{m_{2k}}{2k!} \sim \frac{\sigma^k}{k!}$ which is summable by comparison to the exponential power series. 
	

    Label the $k^{th}$ moments of $\rho_0$ as $m^0_k$. Each $m_{2k}\leq\max(m^0_{2k},m^*_{2k}),\,\forall t$. As the combination of two summable series, this upper bound is summable and, on invoking lemma \ref{norm}, we have a summable upper bound that holds $\forall t$.
\end{proof}

In fact this is sufficient to satisfy the hypotheses of Theorem \ref{convergencethm} with $\theta_0=1$.


\bibliographystyle{abbrv}
\bibliography{ammends4} 
\end{document}